\documentclass[10pt]{amsart}

\usepackage{amsmath,xspace,amssymb,mathrsfs}
\usepackage{color}

\input xy
\xyoption{all}
\xyoption{2cell}
\UseAllTwocells
\CompileMatrices

\newcommand{\Ker}{\operatorname{Ker}}

\newcommand{\Hom}{\operatorname{Hom}}

\newcommand{\Ext}{\operatorname{Ext}}
\newcommand{\Char}{\operatorname{char}}

\newcommand{\Image}{\operatorname{Im}}

\newtheorem{proposition}{Proposition}[section]
\newtheorem{lemma}[proposition]{Lemma}

\newtheorem{corollary}[proposition]{Corollary}
\newtheorem{theorem}[proposition]{Theorem}

\theoremstyle{definition}
\newtheorem{definition}[proposition]{Definition}
\newtheorem{example}[proposition]{Example}

\newtheorem{remark}[proposition]{Remark}

\usepackage{etoolbox}
\makeatletter
\patchcmd{\@settitle}{\uppercasenonmath\@title}{}{}{}
\patchcmd{\@setauthors}{\MakeUppercase}{}{}{}
\makeatother

\begin{document}

\title[Splitting in a complete local ring]{Splitting in a complete local ring and decomposition its group of units}

\author[A. Tarizadeh]{Abolfazl Tarizadeh}
\address{Department of Mathematics, Faculty of Basic Sciences, University of Maragheh, Maragheh, East Azerbaijan Province, Iran.}
\email{ebulfez1978@gmail.com}

\date{}

\begin{abstract} Let $(R,M,k)$ be a complete local ring (not necessarily Noetherian). As the first main result of this article, we prove that in the unequal characteristic case  $\Char(R)\neq\Char(k)$, the natural surjective map between the groups of units $R^{\ast}\rightarrow k^{\ast}$ admits a splitting. Next, we reprove by a new method that $R$ is equi-characteristic, i.e., $\Char(R)=\Char(k)$ if and only if the natural surjective ring map $R\rightarrow k$ admits a splitting, or equivalently, $R$ has a coefficient field. In our proof there is no need for the existence of the coefficient fields for equi-characteristic complete local rings, whose existence is the most difficult part of the known proof. This is one of the main contributions of the article. \\  
As an application, we show that for any complete local ring $(R,M,k)$ the following short exact sequence of Abelian groups: $$\xymatrix{1\ar[r]&1+M\ar[r]& R^{\ast}\ar[r]&k^{\ast}
\ar[r]&1}$$ is always split. In particular, we have an isomorphism of Abelian groups $R^{\ast}\simeq(1+M)\times k^{\ast}$. We also show with an example that the above exact sequence does not split for many incomplete local rings.  
\end{abstract}

\maketitle

\section{Introduction}

In this article, we first show that for any ``finite'' local ring $(R,M,k)$ the following short exact sequence of Abelian groups: $$\xymatrix{1\ar[r]&1+M\ar[r]& R^{\ast}\ar[r]&k^{\ast}
\ar[r]&1}$$ is split. Next, this observation naturally led us to the question of whether the same splitting holds for any ``infinite" local ring. 

Then, after exploring several stages of the level of generalization, first for Artinian local rings, then for local rings whose Jacobson radical are nilpotent, finally as one of the main results of this investigation, we succeeded to show that the same splitting holds for any complete local ring. However, the proof of this general case is technical, and uses the tools and methods of commutative algebra, the theory of fields, and homological algebra. In fact, we prove the following result: 

\begin{theorem}\label{Theorem A} For a complete local ring $(R,M,k)$ the following statements are equivalent: \\
$\mathbf{(i)}$ $R$ has a coefficient field. \\
$\mathbf{(ii)}$ The natural surjective ring map $R\rightarrow k$ admits a splitting. \\
$\mathbf{(iii)}$ $R$ is equi-characteristic, i.e.,  $\Char(R)=\Char(k)$. \\
Furthermore, in the unequal characteristic case  $\Char(R)\neq\Char(k)$, the natural surjective map between the groups of units $R^{\ast}\rightarrow k^{\ast}$ admits a splitting.
\end{theorem}

Note that the first part of this theorem is essentially a reformulation of the most important version of Cohen's structure theorem, but the second part is a new result. The existence of coefficient fields for equi-characteristic complete local rings is the most difficult part of the known proof (see e.g. \cite[Theorem 28.3(ii)]{Matsumura}). However, our proof is quite natural and does not require the existence of the coefficient fields. In fact, unlike the classical route, we prove the splitting $R\rightarrow k$ directly, then the image of that splitting is automatically a coefficient field.  

The characteristic of the residue field $k=R/M$ plays an important role in our proof and hence it is necessary to examine it carefully. We know that in any local ring $(R,M,k)$, if $\Char(R)=\Char(k)$ then there are two cases: $\Char(k)=0$ or $\Char(k)=p$  for some prime number $p$. However, if $\Char(R)\neq\Char(k)$, then the characteristic of $k$ is always a prime number $p$. 

To prove the above theorem in the equi-characteristic case, one of the main ingredients that applied is the following important result:

\begin{lemma}\label{Lemma B} Let $k/F$ be an extension of fields with $F$ a perfect field. If $f:A\rightarrow k$ is a surjection of $F$-algebras with nilpotent kernel, then there is a ring section $g:k\rightarrow A$.   
\end{lemma}

In an attempt to prove the above lemma, we first obtained the following characterization for separability:

\begin{lemma} An algebraic extension of fields is separable if and only if it is formally étale. In particular, in the characteristic zero, every algebraic extension of fields is formally étale.
\end{lemma}

Then the following result paves the way for proving Lemma \ref{Lemma B}, especially in the positive characteristic case. 

\begin{lemma} Every field extension over a perfect field is formally smooth. 
\end{lemma}

To prove Theorem \ref{Theorem A} in the unequal characteristic case, one of the main tools which used is a key result in homological algebra on the vanishing of Ext groups: 

\begin{corollary} Let $R$ be a PID and $M,N$ be modules over $R$. If there is an element $r\in R$ such that $r$ is regular on $M$ and $rN=0$, then $\Ext_{R}^{1}(M,N)=0$.
\end{corollary} 

Other results applied in the proof of Theorem \ref{Theorem A} are Lemmas \ref{Lemma 1} and \ref{Lemma 2}. As an immediate application of this theorem, the desired splitting is deduced:

\begin{corollary} Let $(R,M,k)$ be a complete local ring. Then the following exact sequence of Abelian groups: $$\xymatrix{1\ar[r]&1+M\ar[r]& R^{\ast}\ar[r]&k^{\ast}
\ar[r]&1}$$ is split. 
\end{corollary}

Next, we show that the above exact sequence does not split for most incomplete local rings (see Example \ref{Example 1}). 

We also show with an example that Theorem \ref{Theorem A} is in its full generality, and it cannot be generalized further (see Example \ref{Example excellent 3}). 

Splitting in a finite local ring is presented in Section 2. But splitting in a general complete local ring is presented in \S3.  

\section{Finite case-the initial motivation}

Throughout this article, for a commutative ring $R$ by $R^{\ast}$ we mean the group of units of $R$. If $f:A\rightarrow B$ is a morphism of rings, then the induced map $f^{\ast}:A^{\ast}\rightarrow B^{\ast}$ given by $a\mapsto f(a)$ is a morphism of groups. In this way, we obtain a covariant functor from the category of (commutative) rings to the category of (Abelian) groups. If $f$ is surjective, then $f^{\ast}$  is not necessarily surjective in general. In this regard, we have the following simple but important result (which generalizes \cite[Proposition 2.2]{Del Corso-Dvor}):

\begin{lemma}\label{Lemma 1} If $I$ is an ideal of a ring $R$, then the following statements are equivalent: \\
$\mathbf{(i)}$ $I$ is contained in the Jacobson radical of $R$. \\
$\mathbf{(ii)}$ $1+I$ is a subgroup of $R^{\ast}$. \\
$\mathbf{(iii)}$ We have the following short exact sequence of Abelian groups: $$\xymatrix{1\ar[r]&1+I\ar[r]& R^{\ast}\ar[r]&(R/I)^{\ast}
\ar[r]&1.}$$
\end{lemma}

\begin{proof} (i)$\Rightarrow$(ii): It is clear that $1+I$ is a nonempty subset of $R^{\ast}$ and is also stable under multiplication. All that remains is to show that if $x\in I$ then $(1+x)^{-1}\in 1+I$. There exists some $y\in R$ such that $(1+x)y=1$. It follows that $(1+x)^{-1}=y=1-xy\in 1+I$. \\
(ii)$\Rightarrow$(iii): It can be easily seen that $1+I$ is the kernel of the natural map $R^{\ast}\rightarrow (R/I)^{\ast}$ which is given by $r\mapsto r+I$. This map is also surjective. Indeed, if $r+I\in(R/I)^{\ast}$ then $1-rs\in I$ for some $s\in R$. It follows that $rs\in 1+I\subseteq R^{\ast}$. Thus $rst=1$ for some $t\in R$. This shows that $r\in R^{\ast}$. \\
(iii)$\Rightarrow$(i): If $1+I\subseteq R^{\ast}$, then $I$ is contained in the Jacobson radical of $R$. 
\end{proof}

It can be seen that the characteristic of any ring $R$ with only trivial idempotents is zero or a power of a prime number (local rings as well as rings with a unique minimal prime ideal are typical examples of such rings).
Indeed, if $n=\Char(R)$ is not in the above form, then the prime subring of $R$ which is isomorphic to $\mathbb{Z}/n$ will have non-trivial idempotents (in fact, using Chinese remainder theorem, this ring has precisely $2^{d}$ idempotents where $d$ is the number of distinct primes that appear in the prime factorization of $n$). Then $R$  will also have non-trivial idempotents, which is a contradiction.

Also recall that if $f:A\rightarrow B$ is a morphism of rings, then it can be seen that $\Char(B)$ divides $\Char(A)$, i.e., $\Char(B)|\Char(A)$. In particular, if we have morphisms of rings $A\rightarrow B$ and $B\rightarrow A$, then $\Char(A)=\Char(B)$.
 
If we have a ring map $A\rightarrow B$ with $\Char(B)=0$, then $\Char(A)=0$. However, if $\Char(A)=p^{n}$ for some prime number $p$, then $\Char(B)=p^{d}$ with $d\leqslant n$. 

Keeping these in mind, we come to the following conclusion which is a well known result, its only innovation is probably our new proof: 

\begin{lemma}\label{Lemma 6} The order of any finite local ring is a positive power of a prime number that is the characteristic of its residue field. 
\end{lemma}

\begin{proof} Let $(R,M,k)$ be a finite local ring of order $n$. We know that $R$ is of characteristic $p^{d}$ where $p$ is a prime number and $d\geqslant1$. If $q$ is a prime factor of $n$ then by Cauchy theorem for finite groups (see e.g. \cite[Corollary 5.4.]{Issacs}), the additive group of $R$ has an element $x\in R$ of order $q$. This means that $q\cdot x=0$. We also have $p^{d}\cdot x=0$. Then $q|p^{d}$ and so $q=p$. This shows that $n=p^k$ where $k\geqslant d$. Finally, using the natural ring map $R\rightarrow k$ we conclude that $p=\Char(k)$. 
\end{proof}

In particular, the order of a finite field is a positive power of a prime number. It is also clear that the converse of the above lemma is not true. For example, consider the non-local ring $\mathbb{Z}/p\times\mathbb{Z}/p$ which is of order $p^{2}$.   

\begin{lemma}\label{Lemma 5} Let $G$ be a finite Abelian group of order $|G|=mn$ with $\gcd(m,n)=1$. If $H$ is a subgroup of $G$ of order $m$, then  there exists a subgroup $K$ of $G$ of order $n$ such that $G=HK\simeq H \times K$ and the following short exact sequence: $$\xymatrix{1\ar[r]&H\ar[r]& G\ar[r]&G/H
\ar[r]&1}$$ is split.  
\end{lemma}

\begin{proof} This is an easy exercise in basic group theory. In fact, it can be seen that $H=\{x\in G: x^{m}=1\}$. Indeed, take $x\in G$ with $x^{m}=1$. If $x\notin H$ then $xH$ is a nontrivial element of the group $G/H$ and hence its order $d\geqslant2$. But $|G/H|=n$ and so $d|n$ (because according to Lagrange's theorem, the order of each member divides the order of the group). We also have $(xH)^{m}=H$ and so $d|m$. But this is in contradiction with $\gcd(m,n)=1$. \\ 
Since $G$ is Abelian, $K:=\{x\in G: x^{n}=1\}$ is a subgroup of $G$. We also have $H\cap K=1$, because $\gcd(m,n)=1$. Thus $|HK|=(|H|\cdot|K|)/|H\cap K|=mn$. This shows that $G=HK$. Then the map $G=HK\rightarrow H \times K$ given by $g=hk\mapsto (h,k)$ is an isomorphism of groups, and its composition with the projective map $H\times K\rightarrow H$ is an splitting of the inclusion map $H\rightarrow G$.   
\end{proof}

\begin{remark} We should add that a variant of the above lemma is also true more generally for finite non-Abelian groups, which is known as the Schur-Zassenhaus theorem. But the proof of the general case is technical. However, we only need the Abelian case of this theorem here.
\end{remark}

Recall that if $I$ is an ideal of a ring $R$, then $R$ is complete with respect to the $I$-adic toplogy (i.e., $R$ is Hausdorff and every Cauchy sequence is convergent in $R$) if and only if the natural ring map $R\rightarrow
\lim\limits_{\overleftarrow{n\geqslant1}}R/I^{n}$ given by $r\mapsto(r+I^{n})_{n\geqslant1}$ is an isomorphism.
In the same vein, by complete local ring we mean a local ring $(R,M)$ such that it is complete with respect to the $M$-adic topology. Note that in this definition $R$ is not assumed to be Noetherian (contrary to the classical literature, see e.g. \cite{Cohen}, where complete local rings were also assumed to be Noetherian).

\begin{remark}\label{Remark 5 five} It can be easily seen that every Artinian local ring or more generally every local ring  whose maximal ideal is nilpotent is a complete local ring. In particular, every finite local ring is complete. As another example, if $M$ is a maximal ideal of a ring $R$, then the maximal ideal $M/M^{d}$ of $R/M^{d}$ is nilpotent and hence it is a complete local ring for all $d\geqslant1$.  
\end{remark}

The finite case of one of the main results of this article reads as follows (this finite case was indeed the initial motivation for all subsequent explorations in this article): 

\begin{theorem}\label{Prop 1} Let $(R,M,k)$ be a finite local ring. Then the sequence: $$\xymatrix{1\ar[r]&1+M\ar[r]& R^{\ast}\ar[r]&k^{\ast}
\ar[r]&1}$$ is split-exact. 
In particular, we have an isomorphism of groups $R^{\ast}\simeq (1+M) \times k^{\ast}$.
\end{theorem}

\begin{proof} By Lemma \ref{Lemma 1}, the above sequence is exact. To prove splitting, it can be seen that the order of the (multiplicative) group $1+M$ is coprime to the order of the quotient group $R^{\ast}/(1+M)\simeq k^{\ast}$. Indeed, using Lemma \ref{Lemma 6} and Lagrange theorem (note that $M$ is an additive subgroup of $R$), we have $|1+M|=|M|$ is a power of the prime $p=\Char(k)$, but $|k^{\ast}|=p^{d}-1$ for some $d\geqslant1$. Then apply Lemma \ref{Lemma 5}.
\end{proof}

We also learned that the above finite case has also been investigated in the proof of \cite[Theorem 3.1]{Del Corso-Dvor}. 

\section{The general case}

This section is the technical heart and culmination of this article.  To prove the main results, we need the following lemmas:

\begin{lemma}\label{Lemma 2} Let $I$ be an ideal of a ring $R$ which is contained in the Jacobson radical of $R$. Then for each $n\geqslant1$, the multiplicative group $G=(1+I^{n})/(1+I^{n+1})$ is naturally isomorphic to the additive group $I^{n}/I^{n+1}$, and we have the following short exact sequence of Abelian groups: $$\xymatrix{1\ar[r]&G\ar[r]& (R/I^{n+1})^{\ast}\ar[r]&(R/I^{n})^{\ast}
\ar[r]&1.}$$ In particular, the Abelian group $G$ is vanished by $\Char(R/I)$.  
\end{lemma}

\begin{proof} We know that $1+I^{n}$ is a subgroup of $R^{\ast}$ (see Lemma \ref{Lemma 1}). The map $1+I^{n}\rightarrow I^{n}/I^{n+1}$
given by $1+x\mapsto x+I^{n+1}$ is a morphism of groups, because if $x,y\in I^{n}$ then $xy\in I^{2n}\subseteq I^{n+1}$ and so $x+y+xy+I^{n+1}=x+y+I^{n+1}$. This map is clearly surjective and its kernel is equal to $1+I^{n+1}$. Hence, the above map induces an isomorphism of Abelian groups $G\simeq I^{n}/I^{n+1}$. \\
To see the exactness of the above sequence, it can be seen that the map $1+I^{n}\rightarrow 1+(I^{n}/I^{n+1})\subseteq (R/I^{n+1})^{\ast}$  given by $1+x\mapsto(1+x)+I^{n+1}$ is a surjective morphism of groups whose kernel is equal to $1+I^{n+1}$. Thus this map induces an isomorphism of Abelian groups $G\simeq1+(I^{n}/I^{n+1})$.
Then apply Lemma \ref{Lemma 1}. \\
Using the above isomorphism of Abelian groups $\xymatrix{G\ar[r]^{\simeq\:\:\:\:\:\:\:\:\:\:\:}
&I^{n}/I^{n+1},}$ we can make $G$ into an $R/I$-module (and then this map becomes an isomorphism of $R/I$-modules). Now the statement about the vanishing is clear, because if $M$ is a module over a ring $R$ then $M$ is vanished by $\Char(R)$.   
\end{proof}

Recall that if $R$ is a reduced ring of characteristic $p$ with $p$ a prime number, then the group $R^{\ast}$ has trivial $p$-torsion. Indeed, if $a^{p}=1$ for some $a\in R^{\ast}$, then using the Frobenius endomorphism we have $(a-1)^{p}=a^{p}-1=0$. This shows that $a-1$ is nilpotent, and so $a=1$. In particular, if $k$ is a field of positive characteristic $p$, then $k^{\ast}$ has trivial $p$-torsion.

\begin{remark}\label{Remark 1} Let $\mathscr{C}$ be a category which admits inverse limits. Then for any object $A$ and for any inverse system of objects $(B_{i})_{i\in I}$ in this category, it can be easily seen that the $\Hom_{\mathscr{C}}(A,B_{i})$ is an inverse system of sets over the same index set $I$ and
the natural map:  $$\Hom_{\mathscr{C}}(A,\lim\limits_{\overleftarrow{i\in I}}B_{i})\rightarrow\lim\limits_{\overleftarrow{i\in I}}
\Hom_{\mathscr{C}}(A,B_{i})$$ given by $f\mapsto(p_{i}f)_{i\in I}$
is an isomorphism (bijection) of sets where $p_{k}:\lim\limits_{\overleftarrow{i\in I}}B_{i}\rightarrow B_{k}$ is the projection morphism. 
\end{remark}

Recall that a morphism $f:A\rightarrow B$ in a category has a section if there is a morphism $g:B\rightarrow A$ in that category
such that $fg:B\rightarrow B$ is the identity morphism.

If in a local ring $(R,M,k)$ the natural ring map $R\rightarrow k$ admits a splitting (i.e., has a section) then $\Char(R)=\Char(k)$. We will see that the converse is also true for complete local rings. 

\begin{remark}\label{Remark 3} It can be seen that a ring map $f: R\rightarrow S$ is formally étale (in the sense of \cite[Tag 00UQ]{Johan}) if and only if the infinitesimal lifting property of formally étale maps holds for any nilpotent ideal. Indeed, if $f$ is formally étale, then it suffices to show that any commutative diagram of rings of the following form:  $$\xymatrix{
R\ar[r]^{f}\ar[d]^{}
&S\ar[d]^{} \\A\ar[r]^{}&
A/I}$$
can be fulfilled, i.e., one can find a unique ring map $S\rightarrow A$ that makes the new diagram commute, where $I$ is a nilpotent ideal of $A$ and $A\rightarrow A/I$ is the natural ring map. To find such a ring map, we first extend from $A/I$ to $A/I^{2}$, then to $A/I^{4}$ and so on. More precisely, the following diagram is commutative: $$\xymatrix{
R\ar[r]^{f}\ar[d]^{}
&S\ar[d]^{} \\A/I^{2}\ar[r]^{}&
A/I}$$ where $A/I^{2}\rightarrow A/I$ is the natural ring map. Since the ideal $I/I^{2}\subseteq A/I^{2}$ is of square-zero, there exists a unique ring map $S\rightarrow A/I^{2}$ making the new diagram commute. Then apply the same argument for the following commutative diagram: $$\xymatrix{
R\ar[r]^{f}\ar[d]^{}
&S\ar[d]^{} \\A/I^{4}\ar[r]^{}&
A/I^{2}.}$$
In this way, the desired ring map $S\rightarrow A$ is eventually obtained, because $I^{n}=0$ for some $n\geqslant1$. The same observation is true for formally smooth ring maps. More precisely, a ring map $f: R\rightarrow S$ is formally smooth (in the sense of \cite[Tag 00TI]{Johan}) if and only if the infinitesimal lifting property of formally smooth maps holds for any nilpotent ideal. 
\end{remark}

Formally étaleness is preserved by direct limits as follows:

\begin{lemma}\label{Lemma 9 nine} Let $f: R\rightarrow S$ be a morphism of rings. If $S$ is the direct limit of a  direct system $(S_{i})$ of formally étale $R$-algebras $f_{i}: R\rightarrow S_{i}$ such that the $f_{i}$ are compatible with transition maps and $f$ factors through the $f_{i}$, then $f$ is formally étale. 
\end{lemma}

\begin{proof} Consider the following commutative diagram of rings:  $$\xymatrix{
R\ar[r]^{f}\ar[d]^{h}
&S\ar[d]^{h'} \\A\ar[r]^{\pi}&
A/I}$$ where $I$ is an ideal of $A$ of square-zero and $\pi$ is the natural ring map. Then for each $i$, the following diagram is commutative:  $$\xymatrix{
R\ar[r]^{f_{i}}\ar[d]^{h}
&S_{i}\ar[d]^{h'\alpha_{i}} \\A\ar[r]^{\pi}&
A/I}$$ where $\alpha_{i}:S_{i}\rightarrow S$ is the natural ring map, and so there exists a unique ring map $g_{i}:S_{i}\rightarrow A$ making the new diagram commute. It can be easily seen that the maps $g_{i}$ are compatible with the transition maps of the above direct system. Then by the universal property of direct limits, there exists a unique ring map $g:S\rightarrow A$ such that $g_{i}=g\alpha_{i}$ for all $i$. \\
We show that $g$ is the unique ring map making the above first diagram commute. We may assume that $S$ is a nonzero ring (if $S$ is a zero ring, then $A$ is also a zero ring and the assertion is clear in this case), and hence the index set of the above direct system is nonempty, so we may choose some index $i$. We have $f=\alpha_{i}f_{i}$ and so $gf=g\alpha_{i}f_{i}=g_{i}f_{i}=h$. To see $\pi g=h'$, take $x\in S$, then there exists an index $i$ and some $s\in S_{i}$ such that $x=\alpha_{i}(s)$, then $(\pi g)(x)=\pi g\alpha_{i}(s)=\pi g_{i}(s)=h'\alpha_{i}(s)=h'(x)$. Hence, $g$ makes the above first diagram commute. To prove its uniqueness, suppose $g':S\rightarrow A$ is another ring map making the above first diagram commute. To see $g=g'$, it suffices to show that $g_{i}=g'\alpha_{i}$ for all $i$. We have $g'\alpha_{i}f_{i}=g'f=h$ and $\pi g'\alpha_{i}=h'\alpha_{i}$, and so by the uniqueness of $g_{i}$, we get that $g_{i}=g'\alpha_{i}$ for all $i$. This completes the proof.  
\end{proof}

Recall that by a separating transcendence basis for a given finitely generated field extension of fields $K/k$ we mean a (finite) transcendence basis $\{ x_{i} : i\in I\}\subseteq K$ for this extension such that the finite extension $K/k(x_{i} : i\in I)$ is a separable algebraic extension. 

Also recall that the notion of separable extension can be also defined for arbitrary extension of fields (see e.g. \cite[Tag 030O]{Johan}). For convenience, we mention this definition here: 

\begin{definition}\label{Def 1} An arbitrary extension of fields $K/k$ (not necessarily algebraic) is called separable if for any subextension $k\subseteq K'\subseteq K$ with $K'$ a finitely generated field extension of $k$, the extension $k\subseteq K'$ has a separating transcendence basis. 
\end{definition}

\begin{remark}\label{Remark 4} It can easily be seen that the above definition generalizes the classical notion of separable algebraic extension. Indeed, if an algebraic extension of fields $K/k$ is separable in the above sense, then for any subextension $k\subseteq K'\subseteq K$ with $K'$ a finitely generated field extension of $k$, the extension $k\subseteq K'$ has a separating transcendence basis. Since $K/k$ is an algebraic extension, the extension $k\subseteq K'$ is also algebraic. Hence, the above separating transcendence basis is just the empty set, and so the extension $k\subseteq K'$  is a separable algebraic extension (in the usual sense). As $K'$ varies through the above subextensions, we conclude that the original extension $K/k$ is a separable algebraic extension (in the usual sense). Indeed, if $a\in K$ then by the above observation, $K'=k(a)$ is a separable algebraic extension over $k$. In particular, the element $a$ is separable over $k$.  
\end{remark}  

Note that Definition \ref{Def 1} can in turn be generalized in the following way (although we do not need it in this article): a morphism of rings $k\rightarrow A$ with $k$ a field is called a separable morphism  (in the sense of \cite[\S26]{Matsumura}) if for any field extension $k'/k$ the ring $k'\otimes_{k}A$ is reduced. For field extensions, these two definitions are essentially the same. More precisely, an extension of fields $K/k$ is separable (in the sense of Definition \ref{Def 1}) if and only if it is formally smooth (see \cite[Tags 031Y and 0320]{Johan}), or equivalently, for any field extension $k'/k$ the ring $k'\otimes_{k}K$ is reduced (see \cite[Theorem 26.9]{Matsumura}). In the algebraic extension case, we have the following result:

\begin{lemma}\label{Theorem 3} An algebraic extension of fields $K/k$ is separable if and only if it is formally étale. 
\end{lemma}

\begin{proof} First assume that $K/k$ is a separable algebraic extension of fields. Then we can write $K$ as the filtered union (direct limit) of all finite separable subextensions of $k$, i.e., $K=\bigcup_{\alpha}K_{\alpha}=
\lim\limits_{\overrightarrow{\alpha}}K_{\alpha}$ where each subfield $K_{\alpha}$ of $K$ is a finite separable extension of $k$ (note that the $K_{\alpha}$ are directed by inclusion inside $K$ and the transition maps are indeed the inclusion maps). It is well known that every finite separable extension of fields is formally étale (see \cite[Tags 00U3 and 00UR]{Johan} or \cite[Corollary 4.7.3]{Ford}). Thus for each $\alpha$, the extension $K_{\alpha}/k$ is formally étale. Then by Lemma \ref{Lemma 9 nine}, then original extension $K/k$ is formally étale as well. Conversely, if the extension $K/k$ is formally étale, then by \cite[Tag 031Y]{Johan}, this extension is separable in the above sense (see Definition \ref{Def 1}). Then by Remark \ref{Remark 4}, it is a separable algebraic extension.    
\end{proof}
 
\begin{remark} In relation with the above theorem, note that although any flat and separable morphism of commutative rings is formally étale \cite[Corollary 4.7.3]{Ford}, the notion of separable morphism of rings in the sense of \cite{Ford} only generalizes to ``finite'' separable extensions in fields (see \cite[Corollary 4.5.4(1)]{Ford}). 
\end{remark}

\begin{corollary}\label{Coro 3 three} If $k$ is a field of characteristic zero, then every algebraic extension of $k$ is formally étale. 
\end{corollary}

\begin{proof} In the characteristic zero, every algebraic extension of fields is separable. Then apply Lemma \ref{Theorem 3}.
\end{proof}

To prove the first main theorem we also need to the following result: 

\begin{lemma}\label{Lemma 7} Let $k$ be a field of characteristic zero. If $f:A\rightarrow k$ is a surjection of $\mathbb{Q}$-algebras with nilpotent kernel, then there is a ring section $g:k\rightarrow A$.   
\end{lemma}

\begin{proof} Let $S=\{x_i: i \in I\}\subseteq k$ be a transcendence basis of the extension $k/\mathbb{Q}$. Since the set $S$ is algebraically independent over $\mathbb{Q}$, the natural morphism of $\mathbb{Q}$-algebras from the polynomial ring $\mathbb{Q}[X_{i}: i \in I]$ to $k$ given by $X_{i}\mapsto x_{i}$ is injective. 
Let $R=\mathbb{Q}[x_{i}: i \in I]$ be the image of this ring map. Since $f$ is surjective, for each $i$, there exists some $a_i \in A$ such that $f(a_{i})=x_{i}$. Then again by the universal property of polynomial rings, we have a (unique) morphism of $\mathbb{Q}$-algebras $R\rightarrow A$ which maps each $x_{i}$ into $a_{i}$. \\ 
To extend this map to $k_{0}=\mathbb{Q}(x_i : i \in I)$, the field of fractions of the integral domain $R$, we need to show that the image of every nonzero element $P(x_{i_{1}},\ldots, x_{i_{n}})$ of $R$ under this map is invertible in $A$. 
Note that $A$ is a local ring with the maximal ideal $M=\Ker(f)$, because $M$ is nilpotent. So it suffices to show that $P(a_{i_{1}},\ldots, a_{i_{n}})\notin M$. Since $f$ is a $\mathbb{Q}$-algebra map, we have $f(P(a_{i_{1}},\ldots, a_{i_{n}}))=P(x_{i_{1}},\ldots, x_{i_{n}})\neq0$. This shows that $P(a_{i_{1}},\ldots, a_{i_{n}})$ is invertible in $A$. So we can extend the above ring map to obtain a ring map $h: k_{0}\rightarrow A$ such that $fh:k_{0}\rightarrow k$ is the inclusion map. This, in particular, shows that $A$ is a $k_{0}$-algebra and $h$ is a $k_{0}$-algebra morphism. \\ 
To extend this ring map $h$ to $k$, since $k$ is an algebraic extension over $k_{0}$, then by Corollary \ref{Coro 3 three}, the extension $k/k_{0}$ is formally étale. This observation allows us to show that the natural map $\Hom_{k_{0}}(k,A) \rightarrow\Hom_{k_{0}}(k, k)$ given by $g'\mapsto fg'$
is bijective where the $\Hom_{k_{0}}(-,-)$ is computed in the category of $k_{0}$-algebras. Indeed, if $h':k\rightarrow k$ is a morphism of $k_{0}$-algebras then we have the following commutative diagram of rings: $$\xymatrix{
k_{0}\ar[r]^{inc\:\:}\ar[d]^{h\:\:\:\:\:\:}
&k\ar[d]^{h'} \\A\ar[r]^{f\:\:}&
k.}$$ Then by hypothesis ($f$ is surjective with nilpotent kernel) and by the infinitesimal lifting property for formally étale maps (see Remark \ref{Remark 3}), there exists a unique ring map $h'':k\rightarrow A$ making the new diagram commute. This shows that $h''$ is a morphism of $k_{0}$-algebras and $h'=fh''$. Hence, the above map is bijective. In particular, there exists a (unique) morphism of $k_{0}$-algebras $g:k\rightarrow A$ which extends $h$ and $fg$ is the identity map. This completes the proof.  
\end{proof}

The key point in the preceding proof is formally smoothness of $k$ over $\mathbb{Q}$, which Lemma \ref{Theorem 4} (and Corollary \ref{Coro 4 dort}) will adapt to any perfect field in place of $\mathbb{Q}$.

We also need a special case of the following result which is about the vanishing of Ext groups: 

\begin{lemma}\label{Lemma 3} Let $M$ and $N$ be modules over a ring $R$. Assume there exists some $r\in R$ such that $r$ is regular on $M$ and $rN=0$. If  $\Ext_{R}^{d}(M/rM,N)=0$ or $\Ext_{R}^{d+1}(M/rM,N)=0$ for some $d\geqslant0$, then $\Ext_{R}^{d}(M,N)=0$.
\end{lemma}

\begin{proof} Consider the following short exact sequence of $R$-modules: $$\xymatrix{0\ar[r]&M\ar[r]& M\ar[r]&M/rM
\ar[r]&0}$$ where the map $M\rightarrow M$ is given by $x\mapsto rx$. Then we obtain the following exact sequence: $$\xymatrix{\Ext_{R}^{d}(M/rM,N)\ar[r]&\Ext_{R}^{d}(M,N)
\ar[r]&\Ext_{R}^{d}(M,N)\ar[r]&\Ext_{R}^{d+1}(M/rM,N)}$$ 
where the second map is multiplication by $r$ on the $\Ext_{R}^{d}(M,N)$. But by \cite[Tag 00LV]{Johan}, this map vanishes since $rN=0$, then in either case of the hypothesis, making $\Ext_{R}^{d}(M,N)$ vanish by exactness.
\end{proof}

\begin{corollary}\label{Coro 1} Let $R$ be a PID and $M,N$ be modules over $R$. If there is an element $r\in R$ such that $r$ is regular on $M$ and $rN=0$, then $\Ext_{R}^{1}(M,N)=0$.
\end{corollary}

\begin{proof} It is well known that if $R$ is a PID, then for any two $R$-modules $M$ and $N$ one has $\Ext_{R}^{i}(M,N)=0$ for all $i\geqslant2$. Then apply Lemma \ref{Lemma 3}.
\end{proof}

\begin{remark} The above corollary is quite sufficient for our purpose. However, we should add that this result holds more generally: Let $R$ be a regular Noetherian ring of Krull dimension $n$ and $M,N$ be modules over $R$. If there is an element $r\in R$ such that $r$ is regular on $M$ and $rN=0$, then $\Ext_{R}^{n}(M,N)=0$.
In fact, this is proved by the same argument as the above by applying a theorem of Serre \cite{Serre} which asserts that if $R$ is regular Noetherian of Krull dimension $n$, then for any $R$-modules $M$ and $N$ one has
$\Ext_{R}^{i}(M,N) = 0$ for all $i> n$. 
\end{remark}

\begin{remark}\label{Remark 2} Let $(B_{n})$ be an inverse system of sets indexed by the set of natural numbers. If the transition maps $B_{n+1}\rightarrow B_{n}$ are surjective for all $n$, then for each $d$ it can be easily seen that the projection map $\lim\limits_{\overleftarrow{n\in\mathbb{N}}}
B_{n}\rightarrow B_{d}$ is surjective as well. 
\end{remark}

Now we are ready to prove the first main result of this article: 

\begin{theorem}\label{Theorem 1} Let $(R,M,k)$ be a complete local ring. \\
$\mathbf{(i)}$ If $\Char(k)=0$, then the natural surjective ring map $R\rightarrow k$ admits a splitting. \\
$\mathbf{(ii)}$ If $\Char(k)=p>0$, then the natural surjective group map $R^{\ast}\rightarrow k^{\ast}$ admits a splitting. 
\end{theorem}

\begin{proof} (i): It suffices to show that the natural map: $$\Hom(k,R/M^{n+1})\rightarrow\Hom(k,R/M^n)$$ given by $g\mapsto\pi g$
is surjective for all $n\geqslant1$, where $\pi:R/M^{n+1}\rightarrow R/M^{n}$ is the natural ring map and the $\Hom(-,-)$ is computed in the category of commutative rings. Suppose we have proven this, then using Remarks \ref{Remark 1} and \ref{Remark 2}, the natural map: $$\Hom(k,R)=
\lim\limits_{\overleftarrow{n\geqslant1}}
\Hom(k,R/M^{n})\rightarrow
\Hom(k,k)$$ given by $g\mapsto \pi' g$ 
is surjective as well where $\pi': R\rightarrow k$ is the natural ring map, so we can lift the identity map $k\rightarrow k$ to get the desired splitting. \\
To prove the above surjectivity, pick a ring map $h:k\rightarrow R/M^{n}$ where $n\geqslant1$ is fixed. Let $A:=R/M^{n+1}\times_{R/M^{n}}k$ be the pullback of the natural ring map $\pi:R/M^{n+1}\rightarrow R/M^{n}$ and $h$ in the category of commutative rings. We have then the following commutative diagram of rings: $$\xymatrix{
A\ar[r]^{f'\:\:\:\:\:\:\:}\ar[d]^{f\:\:\:\:\:\:}
&R/M^{n+1}\ar[d]^{\pi} \\ k\ar[r]^{h\:\:\:\:\:\:}&
R/M^{n}}$$
where $f$ and $f'$ are the projection maps. The kernel of $f$ has square-zero. Indeed, it is enough to show that $xy=0$ for all $x=(a+M^{n+1}, a'+M)$ and $y=(b+M^{n+1}, b'+M)$ in $\Ker(f)$. We have $f(x)=a'+M=0$ and so $a'\in M$, but $x\in A$ and so $a+M^{n}=h(a'+M)=0$, this shows that $a\in M^{n}$. Similarly, we have $b'\in M$ and $b\in M^{n}$.
Then $ab\in M^{2n}\subseteq M^{n+1}$
and so $xy=(ab+M^{n+1},a'b'+M)=0$. \\
The map $f$ is surjective, because the natural ring map $R/M^{n+1}\rightarrow R/M^{n}$ is surjective. \\ 
In addition, $A$ is a $\mathbb{Q}$-algebra. To see this, it suffices to show that the image of every nonzero integer $d$ under the natural (injective) ring map $\mathbb{Z}\rightarrow A$ is invertible in $A$. We have $d\cdot1_{A}=(d\cdot1_{R}+M^{n+1},d\cdot1_{R}+M)$. Since $k=R/M$ is a field of characteristic zero, 
$d\cdot1_{R}+M$ is invertible in $k$, and so 
$d\cdot1_{R}+M^{n}=h(d\cdot1_{R}+M)$ is invertible in $R/M^{n}$. Then, by the proof of Lemma \ref{Lemma 1},  $d\cdot1_{R}+M^{n+1}\in(R/M^{n+1})^{\ast}$. Thus $d\cdot1_{A}\in A^{\ast}$. \\
Also, the projection map $f:A\rightarrow k$ is a $\mathbb{Q}$-algebra map, because if $m/d\in\mathbb{Q}$ then $(f\pi_{A})(m/d)=f((m\cdot1_{A})(d\cdot1_{A})^{-1})=
(m\cdot1_{R}+M)(d\cdot1_{R}
+M)^{-1}=(m\cdot1_{k})(d\cdot1_{k})^{-1}=\pi_{k}(m/d)$ where $\pi_{A}:\mathbb{Q}\rightarrow A$ and $\pi_{k}:\mathbb{Q}\rightarrow k$ are the natural ring maps.  \\
Then by Lemma \ref{Lemma 7}, there is a ring map $g:k\rightarrow A$ such that $fg:k\rightarrow k$ is the identity map. Then the composition of the projection map $f':A\rightarrow R/M^{n+1}$ with $g$ gives us the desired ring map $k\rightarrow R/M^{n+1}$. \\
(ii): First note that as $R$ is complete, the natural ring map $R\rightarrow
\lim\limits_{\overleftarrow{n\geqslant1}}R/M^{n}$ is an isomorphism, and thus the induced map between the groups of units is an isomorphism of groups. We also have $(\lim\limits_{\overleftarrow{n\geqslant1}}R/M^{n})^{\ast}=\lim\limits_{\overleftarrow{n\geqslant1}}(R/M^{n})^{\ast}$. Now to prove the assertion, similar to above, it suffices to show that the natural map: $$\Hom_{\mathbb{Z}}(k^{\ast},(R/M^{n+1})^{\ast})\rightarrow
\Hom_{\mathbb{Z}}(k^{\ast},(R/M^{n})^{\ast})$$ given by $g\mapsto\pi^{\ast}g$
is surjective for all $n\geqslant1$ where $\pi:R/M^{n+1}\rightarrow R/M^{n}$ is the natural ring map and the $\Hom_{\mathbb{Z}}(-,-)$ is computed in the category of Abelian groups. Suppose we have proven this, then using Remarks \ref{Remark 1} and \ref{Remark 2}, the natural map: $$\Hom_{\mathbb{Z}}(k^{\ast},R^{\ast})=
\lim\limits_{\overleftarrow{n\geqslant1}}
\Hom_{\mathbb{Z}}(k^{\ast},(R/M^{n})^{\ast})\rightarrow
\Hom_{\mathbb{Z}}(k^{\ast},k^{\ast})$$ is surjective as well, so we can lift the identity map $k^{\ast}\rightarrow k^{\ast}$ to get the desired splitting. To prove the above surjectivity, by general homological algebra, it suffices to show that $\Ext^1_{\mathbb{Z}}(k^{\ast}, G)=0$, where $G= (1+M^{n})/(1+M^{n+1})$. Indeed, using Lemma \ref{Lemma 2}, from the following short exact sequence of Abelian groups: $$\xymatrix{1\ar[r]& G
\ar[r]&(R/M^{n+1})^{\ast}
\ar[r]&(R/M^{n})^{\ast}\ar[r]&1}$$ we obtain the following exact sequence of Abelian groups:  $$\xymatrix{\Hom_{\mathbb{Z}}(k^{\ast},
(R/M^{n+1})^{\ast})\ar[r]&\Hom_{\mathbb{Z}}(k^{\ast},
(R/M^{n})^{\ast})\ar[r]& 
\Ext_{\mathbb{Z}}^{1}(k^{\ast},
G).}$$ 
But $k^{\ast}$ has trivial $p$-torsion. Also, by Lemma \ref{Lemma 2}, the Abelian group $G$ is vanished by $p$. Then by Corollary \ref{Coro 1}, we have $\Ext_{\mathbb{Z}}^{1}(k^{\ast},G)=0$. This completes the proof. 
\end{proof}

As an application of the above theorem, we have the following conclusion:

\begin{corollary}\label{Coro 2} Let $(R,M,k)$ be a complete local ring. Then the following exact sequence of Abelian groups: $$\xymatrix{1\ar[r]&1+M\ar[r]& R^{\ast}\ar[r]&k^{\ast}
\ar[r]&1}$$ is split. 
In particular, we have an isomorphism of Abelian groups $R^{\ast}\simeq (1+M) \times k^{\ast}$.
\end{corollary}

\begin{proof}  This follows from Theorem \ref{Theorem 1}. 
\end{proof}

\begin{corollary} If a local ring $(R,M,k)$ is Artinian or more generally its Jacobson radical is nilpotent, then the following exact sequence of Abelian groups: $$\xymatrix{1\ar[r]&1+M\ar[r]& R^{\ast}\ar[r]&k^{\ast}
\ar[r]&1}$$ is split. In particular, we have an isomorphism of groups $R^{\ast}\simeq (1+M) \times k^{\ast}$.
\end{corollary}

\begin{proof} This follows from Corollary \ref{Coro 2}  and Remark \ref{Remark 5 five}.
\end{proof}

\begin{example}\label{Example 1} We give an example of an (incomplete) local ring $(R,M,k)$ such that the following exact sequence of Abelian groups: $$\xymatrix{1\ar[r]&1+M\ar[r]& R^{\ast}\ar[r]&k^{\ast}
\ar[r]&1}$$ is not split. First note that if $M$ is a maximal ideal of a ring $R$ then the rings $R/M^{n}$ and $R_{M}/M^{n}R_{M}$ are naturally isomorphic for all $n\geqslant0$. \\
Now let $p$ be a prime number. Consider the localization ring $\mathbb{Z}_{P}$ where $P=p\mathbb{Z}$. By the above, $\mathbb{Z}_{P}/P\mathbb{Z}_{P}\simeq\mathbb{F}_{p}
=\mathbb{Z}/p$. Then the natural surjective map between the groups of units $(\mathbb{Z}_{P})^{\ast}\rightarrow \mathbb{F}_{p}^{\ast}$ does not admit a splitting for $p\geqslant5$, because if this map splits then $\mathbb{F}_{p}^{\ast}$ can be viewed as a subgroup of $(\mathbb{Z}_{P})^{\ast}$, but $\mathbb{Z}_{P}$ does not contain roots of unit of order $p-1$. \\
But for $p=2,3$ the map $(\mathbb{Z}_{P})^{\ast}\rightarrow \mathbb{F}_{p}^{\ast}$ admits a splitting. We know that for any prime number $p$, the local ring $\mathbb{Z}_{P}$ is not complete. This shows that the above short exact sequence may be split for some incomplete local rings (in other words, the converse of Corollary \ref{Coro 2} is not true).    
\end{example}

The following result allows us to prove Lemma \ref{Lemma B}, especially in the positive characteristic case.   

\begin{lemma}\label{Theorem 4} Every field extension over a perfect field is formally smooth. 
\end{lemma}

\begin{proof} Let $k/F$ be an extension of fields with $F$ a perfect field. It is well known that every finitely generated field extension over a perfect field has a separating transcendence basis (for its proof see \cite[Theorem 1]{Mac Lane} or 
\cite[Chap. II, \S13, Theorem 31]{Zariski Samuel}). This shows that the extension $F\subseteq k$ is separable (in the sense of Definition \ref{Def 1}), and so by \cite[Tag 0320]{Johan} or by \cite[\S19, Theorem 19.6.1]{Grothendieck}, this extension is formally smooth. As a second proof, combine Theorems 26.3 and 26.9  from Matsumura's book \cite{Matsumura}.
\end{proof}

\begin{corollary} In the characteristic zero, every extension of fields is formally smooth.
\end{corollary}

\begin{proof} Every field of characteristic zero is perfect. Then apply Lemma \ref{Theorem 4}.
\end{proof}

\begin{corollary}\label{Coro 4 dort} Let $k/F$ be an extension of fields with $F$ a perfect field. If $f:A\rightarrow k$ is a surjection of $F$-algebras with nilpotent kernel, then there is a ring section $g:k\rightarrow A$.   
\end{corollary}

\begin{proof} By Lemma \ref{Theorem 4}, the extension $k/F$ is formally smooth. Using this observation and Remark \ref{Remark 3}, we get that the natural map $\Hom_{F}(k,A)\rightarrow
\Hom_{F}(k,k)$ given by $g'\mapsto fg'$ is surjective where the $\Hom_{F}(-,-)$ is computed in the category of $F$-algebras. In particular, the identity map $k\rightarrow k$ can be lifted to get the desired splitting.
\end{proof}

\begin{corollary}\label{Lemma 8} Let $k$ be a field of characteristic $p>0$. If $f:A\rightarrow k$ is a surjection of $\mathbb{F}_{p}$-algebras with nilpotent kernel, then there is a ring section $g:k\rightarrow A$.   
\end{corollary}

\begin{proof} This is a special case of Corollary \ref{Coro 4 dort} (because every finite field is a perfect field).
\end{proof} 

\begin{remark} It is clear that Lemma \ref{Lemma 7} is also a special case of Corollary \ref{Coro 4 dort}. However, the original proof of Lemma \ref{Lemma 7} has the important advantage of using less machinery. Also note that the original proof of Lemma \ref{Lemma 7} no longer works to derive Corollary \ref{Lemma 8}. 
In fact, the only key point related with the characteristic used in proving Lemma \ref{Lemma 7} was that, in the characteristic zero, every algebraic extension of fields is separable. But this does not necessarily hold in the positive characteristic case. For instance,   
a finite extension of the rational function field $\mathbb{F}_{p}(x)$ is not necessarily separable with $p$ a prime number. For example, the finite extension $\mathbb{F}_{p}(x)\subseteq\mathbb{F}_{p}(x^{1/p})$ is inseparable, because the minimal polynomial of $x^{1/p}$ over $\mathbb{F}_{p}(x)$ has repeated roots: $T^{p}-x=(T-x^{1/p})^{p}$.   
\end{remark}

Let $(R,M,k)$ be a local ring. If $\Char(R)=\Char(k)$ then $\Char(k)=0$ or $\Char(k)=p$ for some prime number $p$. However, if $\Char(R)\neq\Char(k)$ then $\Char(k)$ is always a prime number $p$, but $\Char(R)=0$ or $\Char(R)=p^{n}$ with $n\geqslant2$.

The following result is the most general version of Theorem \ref{Theorem 1}. 

\begin{theorem}\label{Theorem 2} Let $(R,M,k)$ be a complete local ring. Then the natural surjective ring map $R\rightarrow k$ admits a splitting if and only if  $\Char(R)=\Char(k)$. In addition, if $\Char(R)\neq\Char(k)$, then the natural surjective map between the groups of units $R^{\ast}\rightarrow k^{\ast}$ admits a splitting.
\end{theorem}

\begin{proof} Using Theorem \ref{Theorem 1},  it remains  only to show that if $\Char(R)=\Char(k)=p$ for some prime number $p$, then the natural surjective ring map $R\rightarrow k$ admits a splitting. To prove this, similar to the proof of Theorem \ref{Theorem 1}(i),  it suffices to show that the natural map: $$\Hom(k,R/M^{n+1})\rightarrow\Hom(k,R/M^n)$$ is surjective for all $n\geqslant1$ where the $\Hom(-,-)$ is computed in the category of commutative rings. \\
Take a ring map $h: k\rightarrow R/M^{n}$ where $n\geqslant1$ is fixed, and let $A:=R/M^{n+1}\times_{R/M^{n}}k$ be the pullback of the natural ring map $R/M^{n+1}\rightarrow R/M^{n}$ and $h$ in the category of commutative rings. Then we show that the projection map $f:A\rightarrow k$ satisfies the conditions of Corollary \ref{Lemma 8}. As proved in Theorem \ref{Theorem 1}(i), it can be seen that $f$ is surjective and its kernel is of square-zero. \\
We also need to show that $A$ is an $\mathbb{F}_{p}$-algebra. Since $\Char(R)=p$, it can be easily seen that $\Char(R/M^{d})=p$ for all $d\geqslant1$. In addition, since the characteristic of a finite direct product of (commutative) rings is indeed the least common multiple (lcm) of the characteristics of the factor rings, and since any extension of rings have the same characteristics, thus $\Char(A)=p$, i.e., $A$ is an $\mathbb{F}_{p}$-algebra. Similarly, as proved in Theorem \ref{Theorem 1}(i), the projection map $f$ is a morphism of $\mathbb{F}_{p}$-algebras. \\
Then by Corollary \ref{Lemma 8} (or by Corollary \ref{Coro 4 dort}), there is a ring map $g:k\rightarrow A$ such that $fg:k\rightarrow k$ is the identity map. Now the composition of the projection map $A\rightarrow R/M^{n+1}$ with $g$ gives us the desired ring map $k\rightarrow R/M^{n+1}$. This completes the proof.  
\end{proof} 

Recall that a local ring $(R,M,k)$ is said to be have a coefficient field if there is a ring map $K\rightarrow R$ (with $K$ a field) such that $K$ maps isomorphically to $k$ under the natural ring map $R\rightarrow k$. Then we have the following result: 

\begin{lemma}\label{Lemma unb 2026} A local ring $(R,M,k)$ has a coefficient field $K$ if and only if the natural ring map $R\rightarrow k$ admits a splitting. In this case,  $R$ is equi-characteristic, i.e., $\Char(R)=\Char(k)$. 
\end{lemma}

\begin{proof} If $R$ has a coefficient field $K$, then we have a ring map $g:K\rightarrow R$ such that $h:=fg$ is an isomorphism where $f:R\rightarrow k$ is the natural ring map. Then $1=hh^{-1}=f(gh^{-1})$ and so we have a section $gh^{-1}:k\rightarrow R$. Conversely, asume there is a ring map $g:k\rightarrow R$ such that $fg: k\rightarrow k$ is the identity map. Then the map $f\circ i: K=\Image(g)\rightarrow k$ is an isomorphism where $i: K\rightarrow R$ is the inclusion (ring) map. Therefore, $K$ is a coefficient field of $R$. 
\end{proof} 

We show that the converse is also true for complete local rings (which completes the proof of Theorem \ref{Theorem A}): 

\begin{corollary} For a complete local ring $(R,M,k)$ the following statements are equivalent: \\
$\mathbf{(i)}$ $R$ has a coefficient field. \\
$\mathbf{(ii)}$ The natural ring map $R\rightarrow k$ admits a splitting. \\
$\mathbf{(iii)}$ $R$ is equi-characteristic.
\end{corollary}

\begin{proof} For (i)$\Leftrightarrow$(ii)  see Lemma \ref{Lemma unb 2026}, and for
(ii)$\Leftrightarrow$(iii) see Theorem \ref{Theorem 2}.
\end{proof}

Next, we use one of our main results to obtain an interesting observation. First, we need the following two lemmas: 

\begin{lemma}\label{Lemma 12 2026} If $I$ is a finitely generated ideal of a ring $R$, then for each $m\geqslant1$ the kernel of the natural ring map $S:=\lim\limits_{\overleftarrow{n\geqslant1}}R/I^{n}
\rightarrow R/I^m$ is exactly $I^{m}S$. In particular, the ring $S$ is complete with respect to the $IS$-adic topology. 
\end{lemma}

\begin{proof} We provide a new and quite elementary proof of this classical result (no need for the exactness of completion on short exact sequences or Artin-Rees style lemmas): it is clear that for any ideal $I$ (not necessarily finitely generated), then $I^{m}S$ is contained in the kernel of the above map. We prove the reverse inclusion using induction on $m$. To see the induction step ($m=1$), take $x=(a_n +I^{n})\in\Ker(S\rightarrow R/I)$. Then $a_{1}\in I=(f_{1},\cdots, f_k)$ and so we may write $a_{1}=\sum\limits_{i=1}^{k}f_{i}c_{i,1}$ where $c_{i,1}\in R=I^{0}$ (note that if $r\in R$ then by $r$ in $S$ we always mean $(r+I^{n})_{n\geqslant1}$). For each $d\geqslant1$, we have $a_{d+1}-a_{d}\in I^{d}=f_{1}I^{d-1}+\ldots+f_{k}I^{d-1}$ and so $a_{d+1}-a_{d}=\sum\limits_{i=1}^{k}f_{i}c_{i,d}$ where $c_{i,d}\in I^{d-1}$. For each $i$ ($1\leqslant i\leqslant k$) and for each $n\geqslant1$ setting $b_{i,n}:=c_{i,1}+\sum\limits_{d=1}^{n}c_{i,d}$. Then it is clear that $y_{i}:=(b_{i,n}+I^{n})_{n\geqslant1}\in S$ for all $i$, and in $S$ we have $x=\sum\limits_{i=1}^{k}f_{i}y_{i}\in IS$. Now assume the assertion holds for a given $m\geqslant1$ (the induction hypothesis), we prove it for $m+1$. Take $x\in\Ker(S\rightarrow R/I^{m+1})$ then $x\in\Ker(S\rightarrow R/I^{m})=I^{m}S$. Since $I$ is finitely generated, $I^{m}=(g_{1},\ldots,g_{t})$ is also finitely generated. Then we may write $x=\sum\limits_{i=1}^{t}g_{i}s_{i}$ where $s_{i}\in S$. Since $S/IS\simeq R/I$, for each $i$ there exists some $r_{i}\in R$ such that $s_{i}-r_{i}\in IS$. It follows that $x-\sum\limits_{i=1}^{t}g_{i}r_{i}=
\sum\limits_{i=1}^{t}g_{i}(s_{i}-r_{i})\in I^{m+1}S\subseteq\Ker(S\rightarrow R/I^{m+1})$. Then $\sum\limits_{i=1}^{t}g_{i}r_{i}\in\Ker(S\rightarrow R/I^{m+1})$ and so $\sum\limits_{i=1}^{t}g_{i}r_{i}\in I^{m+1}$. This shows that $x\in I^{m+1}S$. This completes the induction argument. \\
Thus we have natural isomorphisms of rings $S/I^{n}S\simeq R/I^{n}$ for all $n\geqslant1$. Then we obtain an isomorphism of rings $g:\lim\limits_{\overleftarrow{n\geqslant1}}S/(IS)^{n}
=\lim\limits_{\overleftarrow{n\geqslant1}}S/I^{n}S
\rightarrow\lim\limits_{\overleftarrow{n\geqslant1}} R/I^{n}=S$. It can be easily seen that $gf$ is the identity map where $f: S\rightarrow
\lim\limits_{\overleftarrow{n\geqslant1}}S/(IS)^{n}$ is the natural ring map. Then $f=g^{-1}$. This shows that $S$ is complete with respect to the $IS$-adic topology.
\end{proof}

We also provide a simple proof for the following known result: 

\begin{lemma} If $M$ is a maximal ideal of a ring $R$ then $S=\lim\limits_{\overleftarrow{n\geqslant1}}R/M^{n}$ is  a local ring with the maximal ideal $M':=\Ker(S\rightarrow R/M)$ and with the residue field $R/M$.
\end{lemma}

\begin{proof} To see locality, take $x:=(a_n +M^n) \in S\setminus M'$. Then $a_n \in R\setminus M$ for all $n\geqslant1$. Since $R/M^n$ is a local ring with the maximal ideal $M/M^n$, there exists some $b_n \in R$ such that $1-a_n b_n \in M^n$ for all $n$. Then $xy=1$ where $y:=(b_n +M^n)$. To complete the proof we need to show that $y\in S$. Since $M$ is a maximal ideal, so $M^n$ is a primary ideal of $R$ for all $n\geqslant1$. We also have $a_n \notin\sqrt{M^n}=M$. Hence, to prove $b_{n+1}-b_n \in M^n$, it suffices to show that $a_n(b_{n+1}-b_n) \in M^n$. We have $a_n(b_{n+1}-b_n)=(a_n -a_{n+1})b_{n+1}+(a_{n+1}b_{n+1}-1)+(1-a_n b_n) \in M^n$ by considering that $a_{n+1}b_{n+1}-1 \in M^{n+1}\subseteq M^n$. This shows that $S$ is a local ring with the maximal ideal $M'$. It is clear that $S/M'\simeq R/M$.
\end{proof}

Now, the observation reads as follows:

\begin{corollary} If $M$ is a finitely generated maximal ideal of a ring $R$ and $S=\lim\limits_{\overleftarrow{n\geqslant1}}R/M^{n}$, then the following exact sequence of Abelian groups: $$\xymatrix{1\ar[r]&1+MS\ar[r]& S^{\ast}\ar[r]&(R/M)^{\ast}
\ar[r]&1}$$ is split. In particular, we have an isomorphism of groups $S^{\ast}\simeq (1+MS) \times (R/M)^{\ast}$.
\end{corollary}

\begin{proof} By the above lemmas, $S$ is a complete local ring with the maximal ideal $MS$ and with the residue field $R/M$. Then apply Corollary \ref{Coro 2}.
\end{proof}

In particular, if $p$ is a prime number then in the ring of $p$-adic integers $\mathbb{Z}_{p}=
\lim\limits_{\overleftarrow{n\geqslant1}}\mathbb{Z}/p^{n}$ we have $\mathbb{Z}_{p}^{\ast}\simeq(1+p\mathbb{Z}_{p})
\times\mathbb{F}_{p}^{\ast}$.

We observed that if $M$ is a finitely generated maximal ideal of a ring $R$, then $S=\lim\limits_{\overleftarrow{n\geqslant1}}R/M^{n}$ is a complete local ring with the maximal ideal $M'=\Ker(S\rightarrow R/M)$. However, if $M$ is not finitely generated, then the local ring $(S,M')$ is not necessarily complete with respect to the $M'$-adic topology (see e.g. \cite[Tag 05JA]{Johan}).

\begin{example}\label{Example excellent 3} We show with an example that Theorem \ref{Theorem 2} cannot be generalized further to non-maximal ideals. More precisely, we first give an example of an ideal $I$ of a (complete local) ring $R$ such that $R$ is complete with respect to the $I$-adic topology and $\Char(R)=\Char(R/I)$, but the natural surjective ring map $R\rightarrow R/I$ does not split. Next, we give an example of an ideal $I$ of a (complete local) ring $R$ such that $R$ is complete with respect to the $I$-adic topology and $\Char(R)\neq\Char(R/I)$, but the natural surjective group map $R^{\ast}\rightarrow (R/I)^{\ast}$ does not split. \\
First of all, note that if a ring $R$ is complete with respect to the $I$-adic topology for some ideal $I$ of $R$, then it can be seen that $I$ is contained in the Jacobson radical of $R$. In particular, if a ring $R$ is complete with respect to the $M$-adic topology for some maximal ideal $M$ of $R$, then $R$ is a complete local ring, and for complete local rings we proved that the above splittings already exist. Hence, to find the above examples, the ideal $I$ must be a non-maximal ideal. \\
To this end, take the ideal $I=(t^{2})/(t^{3})$ in the (complete local) ring $R= k[t]/(t^{3})$ with $k$ a field. It is clear that $R/I\simeq k[t]/(t^{2})$ and
$\Char(R)=\Char(k)=\Char(R/I)$. Also $R$ is complete with respect to the $I$-adic topology, since $I^{2}=0$. But the natural surjective ring map $R=k[t]/(t^{3})\rightarrow R/I=k[t]/(t^{2})$ does not split, because the element $t+(t^{2})$ in $R/I$ has no lifting (in $R$) with square zero. \\ 
As another example from the first type, take (the complete local) ring $R=k[[t]]$ and $I=(t^{2})$. Then $R$ is complete with respect to the $I$-adic topology and $\Char(R)=\Char(R/I)$, but the natural ring map $R\rightarrow R/I$ does not split. \\
For the second type of example, for a prime number $p$, consider the ring of $p$-adic integers $\mathbb{Z}_{p}$ which is a complete local ring of mixed characteristic, i.e., $\Char(\mathbb{Z}_{p})=0$ but the reside field of $\mathbb{Z}_{p}$ is $\mathbb{F}_{p}$, and take the ideal $I=p^{2}\mathbb{Z}_{p}$ where by $p^{2}$ we mean the element $(p^{2}+p^{n}\mathbb{Z})_{n\geqslant1}$ of $\mathbb{Z}_{p}$. By Lemma \ref{Lemma 12 2026}, we have $\mathbb{Z}_{p}/I \simeq\mathbb{Z}/p^2$.
Then the group of units of $\mathbb{Z}_{p}/I$ is of order $p(p-1)$ and by Cauchy theorem has an element of order $p$ (it can be also directly seen that $1+p$ is an element of $(\mathbb{Z}/p^2)^{\ast}$ of order $p$). But for $p\geqslant3$ and for each $n\geqslant1$, by Gauss theorem, $(\mathbb{Z}/p^n)^{\ast}$ is a cyclic group of order $p^{n-1}(p-1)$ and so there is a natural isomorphism of groups $(\mathbb{Z}/p^n)^{\ast}\simeq\mathbb{Z}/(p-1)
\times\mathbb{Z}/p^{n-1}$, thus the group $(\mathbb{Z}_p)^{\ast}=
\lim\limits_{\overleftarrow{n\geqslant1}}
(\mathbb{Z}/p^{n})^{\ast}\simeq\mathbb{Z}/(p-1)
\times(\lim\limits_{\overleftarrow{n\geqslant1}}
\mathbb{Z}/p^{n-1})
\simeq\mathbb{F}_{p}^{\ast}
\times\mathbb{Z}_p$ has no nontrivial $p$-torsion, so there is no splitting in this case.    
\end{example}

\textbf{Acknowledgments.} The author would like to give sincere thanks to Professors Bhargav Bhatt, Brian Conrad, Pierre Deligne, and Melvin Hochster who generously shared their very valuable and excellent ideas with us during the writing of this article.

\end{document}